\documentclass[preprint,12pt]{elsarticle}

\usepackage{epic,eepic,epsfig}

\usepackage{amssymb}
\usepackage[utf8]{inputenc}
\usepackage{xcolor}
\usepackage{xcolor}
\usepackage{lineno} 
\usepackage{xcolor}
\usepackage{graphicx,multicol}

\usepackage{amsmath} 

\usepackage{cases}
\usepackage{tikz}
\usepackage{tkz-graph}
\usepackage{mathtools}

\usepackage{graphicx}
\usepackage{xcolor} 
\usepackage{amsmath}
\usepackage{amssymb}
\usepackage{amssymb}
\usepackage{amsthm}
\usepackage{forloop} 
\usepackage{fp}
\usepackage{lineno}

\usepackage{complexity}

\usepackage{subcaption}

\usepackage{etoolbox}

\definecolor{forestgreen}{rgb}{0.13, 0.55, 0.13}

\newtheorem{theorem}{Theorem}[section]

\newtheorem{lemma}[theorem]{Lemma}

\newtheorem{proposition}[theorem]{Proposition}

\newtheorem{observation}[theorem]{Observation}

\newtheorem{corollary}[theorem]{Corollary}

\newtheorem*{claim*}{Claim}

\journal{Theoretical Computer Science}

\begin{document}

\title{Hunting a rabbit: complexity, approximability and some characterizations
\footnote{A preliminary version of this paper has appeared in the proceedings of the  31st International Computing and Combinatorics Conference, COCOON 2025, Chengdu, China, August 15-17, 2025, LNCS 15983, Springer 2026, ISBN 978-981-95-0214-1}}

\author[1]{Walid Ben-Ameur}
\ead{walid.benameur@telecom-sudparis.eu}

\author[2,3]{Harmender Gahlawat}
\ead{harmendergahlawat@gmail.com}

\author[1]{Alessandro Maddaloni}
\ead{alessandro.maddaloni@telecom-sudparis.eu}

\address[1]{SAMOVAR, Télécom SudParis, Institut Polytechnique de Paris, 91120 Palaiseau, France}

\address[2]{LIMOS, Universit\'e Clermont Auvergne, CNRS, Clermont Auvergne INP, France}
\address[3]{Ben-Gurion University of the Negev, Beersheba, Israel}

\begin{abstract}
In the \emph{Hunters and Rabbit} game, $k$ hunters attempt to shoot an invisible rabbit on a given graph~$G$. In each round, the hunters select $k$ vertices to shoot at, while the rabbit moves along an edge of~$G$. The hunters win if, at any point, the rabbit is shot. 
The \emph{hunting number} of~$G$, denoted $h(G)$, is the minimum integer~$k$ such that $k$ hunters have a winning strategy, regardless of the rabbit’s moves. \\
The complexity of computing $h(G)$ has been the longest standing open problem concerning the game and has been posed as an explicit open problem by several authors. Our first main contribution resolves this question by proving that computing~$h(G)$ is NP-hard, even for bipartite simple graphs. We further show that the problem remains NP-hard even when $h(G) = O(n^{\varepsilon})$ or when $n - h(G) = O(n^{\varepsilon})$, where~$n$ is the order of~$G$. In addition, we prove that it is NP-hard to approximate~$h(G)$ additively within $O(n^{1-\varepsilon})$. \\
When a time limit~$l$ is imposed on the hunting process, we establish that computing~$h(G)$ remains NP-hard for any $l \ge 2$ bounded by a polynomial in~$n$. On the positive side, we present a polynomial-time $l$-factor approximation algorithm for computing the hunting number with time limit~$l$, and we show that~$h(G)$ can be computed in polynomial time for bipartite graphs when only two time slots are allowed ($l = 2$). \\
Finally, we provide a forbidden-subgraph characterization of graphs with loops satisfying $h(G) = 1$, thereby extending a known characterization for simple graphs.
\end{abstract}
\maketitle

\section{Introduction}
The hunters and rabbit game has been studied under several different names. Although we hold nothing against rabbits, we choose to use the terminology of hunters and rabbit, as it is the most widely adopted. 
This game is played on an undirected graph $G$ with a positive integer $k$ representing the number of hunters. In each round or time step, the hunters shoot at $k$ vertices, while the rabbit occupies a vertex unknown to the hunters (until the rabbit is possibly shot). The rabbit can start the game in a given subset of vertices (usually all of them) and, if the rabbit is not shot, it must move to an adjacent vertex after each round. The rabbit wins if it can ensure that its position is never shot; otherwise, the hunters win.

As an example consider a complete graph on $n$ vertices: here $n-1$ hunters can shoot at the same $n-1$ vertices for two rounds and be sure the rabbit will be shot. On the other hand, on a path on $n>2$ vertices $v_1,...,v_n$, one hunter can win by
 subsequently shooting at all vertices from $v_2$, to $v_{n-1}$ and then restart shooting backward $v_{n-1},...,v_2$. 
 

The hunters and rabbit game was introduced in \cite{journals/combinatorics/BritnellW13} for the case $k=1$, where the authors show that one hunter wins on a tree if and only if it does not contain a $3$-spider ($H_1$ in Figure \ref{fig:forbid}) as a subgraph. Further, it was also shown that when one hunter can win, he can win in a number of rounds linear in the number of vertices. Very similar results were obtained also in \cite{HASLEGRAVE20141}.
The minimum number of hunters needed to win on a given graph $G$ is called the \textit{hunting number} of $G$, denoted $h(G)$. In \cite{ABRAMOVSKAYA201612} it is proven that the hunting number is upper bounded by the graph pathwidth plus $1$. 
It is also shown that the hunting number of an $(n\times m)$-grid is $\lfloor \frac{\min(n,m)}{2}\rfloor+1$ and that the hunting number of trees is $O(\log (n))$. On the other hand, there are trees for which the hunting number is $\Omega(\log (n))$ \cite{gruslys2015catching}. In \cite{BOLKEMA2019360} the hunting number of the $n$ dimensional hypercube is proven to be $1+\sum_{i=1}^{n-2}\binom{i}{\lfloor \frac{i}{2} \rfloor}$. Here the authors also show that {the graph degeneracy given by} $\max_{S \subseteq V(G)} \delta(D[S])$ {(i.e., the maximum through induced subgraphs of their minimum degree)} is a lower bound for $h(G)$.  
In \cite{dissaux:hal-03995642} the authors provide a polynomial-time algorithm to determine the hunting number of split graphs. They also show that computing the hunting number on any graph is FPT parameterized by the size of a vertex cover. Moreover they prove that, if a monotone capture is required, the number of hunters must be at least the pathwidth of the graph and it is not possible to additively approximate the monotone hunting number within $O(n^{1-\epsilon})$. 

Another variant  of hunters and rabbit is considered in \cite{althoetmar2025,Bern22}  where a loop is  added around each vertex (so the rabbit is not required to move at each time step). This was called the firefighting game. Notice that an NP-hardness proof of this variant is independently proposed in the recent paper \cite{althoetmar2025}.

{A more general version of the hunters and rabbit problem  was also considered in \cite{ourpaper,benameur2024complexityresultscopsrobber} where the rabbit moves along the edges of a directed graph $D$ that might also contain loops.  
It is shown there that it is NP-hard to decide whether the hunting number of a digraph is $1$. Computing the hunting number is  proved to be FPT parametrized in some generalization of the vertex cover. When the digraph is a tournament, tractability is achieved with respect to the minimum size of a feedback vertex set. The hunting number is also proved to be less than or equal to $1$ + the directed pathwidth. An easy to compute  lower bound  is given by $\max_{S \subseteq V} \max(\delta^+(D[S],\delta^-(D[S])$ (i.e., the maximum through all induced subgraphs of the minimum indegree and the minimum outdegree). When a monotone capture is assumed,  the hunting number is proved to be greater than or equal to the directed pathwidth, while pathwidth plus $1$ is still a valid upper bound. Another result worth mentioning  from \cite{ourpaper} is related to the minimum number of shots (regardless of  the number of hunters) required to shoot the rabbit. It is proved that this number is easy to compute, and that the rabbit can always be shot before time step $n$ using this minimum number of shots. Some connections with the no-meet matroids of \cite{BENAMEUR2022}, as well as with the matrix mortality problem are also drawn in  \cite{benameur2024complexityresultscopsrobber}. }

The hunters and rabbit game falls within the broader category of cops and robber games, where different versions are defined based on factors such as the available moves for the cops, the robber's speed, and the robber's visibility to the cops. These kind of games are widely studied, for a review see the book \cite{Bontato}. The first cops and robber game was defined in \cite{Quillot} and \cite{NOWAKOWSKI1983235}. The difference between the hunters and rabbit and this game is that the cops must follow the edges of the graph and can always see the robber. Deciding whether $k$ cops can catch a robber in this version is EXPTIME-complete when $k$ is part of the input \cite{KINNERSLEY2015201}, but it is polynomial when $k$ is fixed.
 Cops and robber games variants provide algorithmic interpretations of several graph (width) measures like treewidth \cite{SEYMOUR199322}, pathwidth \cite{Parsons1978PursuitevasionIA}, directed pathwidth \cite{Barat}, directed tree-width \cite{seymourdirtw}, DAG-width \cite{10.5555, bang2016dag}. These games have been intensively studied also due to their applications in numerous fields such as multi-agent systems \cite{AlejandroIsaza}, robotics \cite{chungrobots}, database theory \cite{GOTTLOB2003775}, distributed computing \cite{Nisse2019}.

\smallskip
\noindent\textbf{Contributions and paper organization.} \\
In this paper we deal with the hunters and rabbit game on undirected graphs and we also consider the case when those graphs contain loops. Our main result states that computing $h(G)$ is NP-hard. This confirms the sentiment emerging from the literature (e.g. in \cite{ABRAMOVSKAYA201612,dissaux:hal-03995642}). {In particular, we first provide a reduction from \textsc{3-partition} to {the problem of computing the minimum number of hunters, $h_S(G)$, required when the rabbit starts in a subset $S$}. We then provide a polynomial time reduction from computing $h_S(G)$ to computing $h(G)$ for graphs that may contain loops. We conclude by providing a reduction from $h(G)$ on graphs with loops to bipartite graphs.} We also prove that the problem remains hard even if $h(G)$ or $n-h(G)$ is as small as $O(n^{\epsilon})$, for any $\epsilon>0$. Approximating the hunting number of a graph within an additive error of $O(n^{1-\epsilon})$ is shown to be NP-hard too, for every $\epsilon >0$. 
We also consider the variant where rabbit capture is required no later than some time $l$. We show that computing the minimum number of hunters (denoted by $h(G,l)$)  is NP-hard for every $l\ge 2$ that is upper bounded by a polynomial in $n$ and we give an $l$-factor approximation algorithm. Furthermore, we show that $h(G,2)$ is equal to the matching number of $G$ when the graph is bipartite. 
Finally, we extend to graphs that can contain loops, the characterization from \cite{journals/combinatorics/BritnellW13,HASLEGRAVE20141}
of graphs $G$ such that $h(G)=1$.

The paper is organized as follows. Notation and definitions are provided  in Section 2.  Some preliminary properties are presented in Section 3.  Section 4 contains the NP-hardness proof of the hunting number, in addition to  some inapproximability results.
Then the time limit variant is studied in Section 5. Section 6 is devoted to the characterization  of graphs (with loops) where one hunter wins.  Finally, the paper concludes with a few remarks in Section 7.

\section{Notation and some terminology}
\label{sec:nota}
Let us start with some notation. $G=(V,E)$ is an undirected graph where $V(G):= V$ (resp. $E(G):=E$) denotes the set of vertices (resp. edges) of $G$. 
{Unless specified otherwise, we assume throughout this paper that the graphs under consideration contain neither isolated vertices nor parallel edges.}
One can then use $uv$ to denote an edge whose endpoints are $u$ and $v$. When it does not contain loops, the graph is said to be simple.
The number of vertices is generally denoted by $n(G):=|V|$, or simply $n$  when clear from the context.
Let $K_n$ be the complete graph on $n$ vertices without loops, while ${K}^{\circ}_n$ denotes the complete graph with loops (so the number of edges is $n(n+1)/2$).
Given $A \subset V$, $G[A]$ is the subgraph of $G$ induced by $A$.  A stable set (resp. clique) $S \subset V$ is a subset of vertices such that $G[S]$ does not contain edges (resp. is a complete graph). 
Two disjoint subsets of vertices $A$ and $B$ are said to be \emph{fully connected} if each vertex of $A$ is adjacent to each vertex in $B$. For a vertex $v$, let $N(u):= \{w~|~uw \in E(G)\}$ and $N[u]:= N(u) \cup \{u\}$.
Two vertices $u$ and $v$ are said to be \textit{twins} if $N[u] = N[v]$.
Given a subset of vertices $S$, $N(S)$ denotes the set of vertices having at least one neighbor in $S$. 
We will use $[n]$ to denote the set of numbers $\{1,...,n \}$.

{The hunters and rabbit game is played on a graph $G=(V,E)$. We identify with $W_t \subseteq V$ the positions at which hunters are 
\emph{shooting} at time $t$,
 while  $R_t$ denotes the \emph{rabbit territory} (i.e., the set of possible positions of the {invisible} rabbit assuming that he was not yet shot).  We have that $R_1=V\setminus W_1$ and, for $t\ge 2$, $R_t = N(R_{t-1}) \setminus W_t$. A \emph{hunter strategy} {$(W_t)_{t \geq 1}$} 
 is a  \emph{winning strategy}, if and only if, {there exists a finite $T$ such that}  $R_T = \emptyset$. {With a slight abuse of notation, such a winning strategy could be simply identified with the first $T$ sets $W_1,..., W_T$.}
 Observe that $(W_t)_{t \geq 1}$ 
 is not a winning strategy, if and only if, there exists an \emph{escape walk} of the rabbit (i.e., $(v_t)_{t \geq 1}$ such that $v_t \notin W_t$, and $v_t v_{t+1} \in E(G)$, $\forall t \ge 1$).
If $R_t \cap A =\emptyset$, we say that $A$ is \emph{decontaminated} at time $t$. When {$k=\max_{t \geq 1} |W_t|$, we say that the strategy $(W_t)_{t \geq 1}$} uses $k$ hunters. The minimum integer $k$ such that there exist an integer $T$ and a  winning  strategy $W_1,...,W_T$ {such that $k=\max_{t \in [T]} |W_t|$} is called the \emph{hunting number} of $G$ and denoted by $h(G)$.  When the set of possible initial positions of the rabbit is restricted to some subset $S$, the hunting number is noted $h_S(G)$ (so $ h(G) = h_V(G)$).  Note that $R_1 = S \setminus W_1$. If $h(G) \leq k$, $G$ is said to be a \emph{$k$-hunterwin} graph. }

\section{Preliminary results}
\label{sec:prel}
Let $k$ represent the number of hunters ($|W_t| = k$, $\forall t \geq 1$) to be dealt with. From the definition of $h_S(G)$, we can ensure that if $k<h_S(G)$, then there are no  winning strategies using $k$ hunters if the rabbit starts at $S$. In other words, $R_t \neq \emptyset$ for any time $t$.  
The first lemma provides a reciprocal view that is mainly due to the undirected nature of the graph: using $k$ hunters with $k <h_S(G)$ and assuming that the rabbit can start anywhere (not only in $S$), then the rabbit territory will always intersect with $S$ (i.e., $S$ cannot be decontaminated). {Lemma \ref{nodeconaminationS} will be crucial to lift our hardness result  for computing $h_S(G)$ to hardness of computing $h(G)$ in general.}

\begin{lemma}\label{nodeconaminationS}
If $h_S(G) > k$, then any strategy using $k$ 
{hunters} in $G$ is such that $R_{\tau} \cap S \neq \emptyset, \forall \tau \geq 1$.
\end{lemma}
\begin{proof}
Let  $(W_t)_{t \geq 1}$ be any strategy in $G$ using $k$ {hunters}. 
Let $\tau \geq 1$ be an integer and assume hunters shoot at $W_{\tau},...,W_1$ in this order: since $h_S(G) > k$, there must exist a rabbit walk $v_1,...,v_{\tau}$ with $v_1\in S$ that survives against $W_{\tau},...,W_1$ (i.e. $v_t\notin W_{\tau+1-t}$ for $t \in [\tau]$). But then $v_{\tau},...,v_1$ is a rabbit walk to $S$ that survives against $W_1,...,W_{\tau}$, namely  $v_{\tau+1-t}\notin W_t$ for $t \in [\tau]$ and $R_{\tau} \cap S \neq \emptyset$.
    \end{proof}

Let $G = (V,E)$ be a graph whose edge set might contain loops in addition to regular edges.   Let $B_G$ be the undirected  bipartite graph built from $G$ as follows: $V(B_G) = V \cup V'$ where $V'$ is a copy of $V$ and $E(B_G)$ contains edges $v'w$ and $w'v$ for any edge  $v w \in E$. A loop  $v v$ of  $E$ is then represented by one edge $v'v$ in $B_G$. $B_G$ can also be seen as the tensor product of $G$ with $K_2$. 
Observe that $B_G$ does not contain loops (see Figure \ref{fig:BC} for illustration). Next lemma states that $G$ and $B_G$ have the same hunting number. 

\begin{lemma}
$h(G) = h(B_G)$.
\label{lem:bipartite}
\end{lemma}
\begin{proof}
Let $W_1, W_2,..., W_T$ be a winning strategy in  $G$ using $h(G)$ hunters. A winning strategy $W^B$ of length $2T$ is built in $B_G$ as follows: for $1 \leq t \leq T$, let $W^B_t = \{v: v \in W_t \}$ if $t$ is odd and  $W^B_t = \{v': v \in W_t\}$ for even time $t$. For $T+1 \leq t \leq 2T$, we take $W^B_t = \{v: v \in W_{t-T}\}$ if $t$ is even and  $W^B_t = \{v': v \in W_{t-T}\}$ otherwise. If the rabbit was initially in $V$, then he will be shot during the first $T$ iterations, otherwise  this occurs between $T+1$ and $2T$. We consequently have $h(B_G) \leq h(G)$.\\
Conversely, if $(W^B_t)_{t \geq 1} $ is a winning strategy in $B_G$, then by simple projection on $V$ we get a winning strategy in $G$: let $W_t = \{v: v \in W^B_t\} \cup \{v: v' \in W^B_t\}$. This implies that $h(B_G) \leq h(G)$.  \end{proof}

\begin{figure}[htbp]
\centering
\vspace{-15mm}
\includegraphics
[scale=0.35]{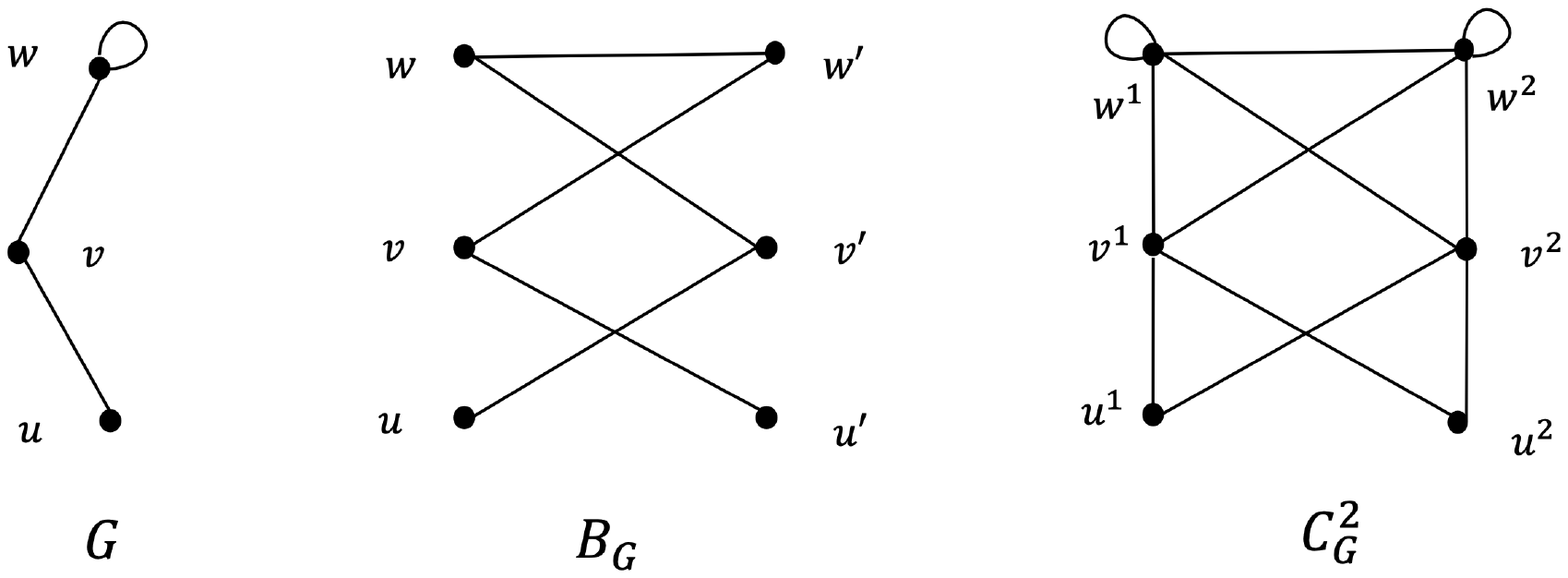} \vspace{-15mm}
    \caption{Illustration of $B_G$ and $C_G^p$ (with $p = 2$)}
\label{fig:BC}
\end{figure}

For some number $p \geq 2$ and some graph $G=(V,E)$, let us build the graph $C_G^p$ as follows:  $p$ copies   $G^i = (V^i,E^i)_{1 \leq i \leq p}$ of  $G$ are considered; if some edge $vw \in E$ then $C_G^p$ contains also edges $v^i w^j$ for $1\leq i, j \leq p$ (see Figure \ref{fig:BC} for illustration).  
The graph $C_G^p$ can be seen as the tensor product of $G$ with ${K}^{\circ}_p$. We show that $h(C_G^p)$ is simply $p \times h(G)$.

\begin{lemma}
$h(C_G^p) = p \times h(G)$.
\label{lem:C}
\end{lemma}
\begin{proof}
Consider a winning strategy $(W_t)_{t \geq 1}$ in $G$ using $h(G)$ hunters. We build a  strategy $(W^C_t)_{t \ge 1}$  in $C_G^p$ using $p \times h(G)$ hunters as follows: $W^C_t = \bigcup_{1 \leq i \leq p} \{v^i: v \in  W_t\}$. This strategy is obviously a winning one showing that $h(C_G^p) \leq p \times h(G) $.
Assume that $h(C_G^p) <p \times h(G)$. Consider a winning strategy $(W^C_t)_{t\ge 1}$ using $h(C_G^p)$ hunters. Observe that at each time step $t$, there exists at least one index $i$ (denoted by $i(t)$) such that  $|W^C_t \cap V^i| < h(G)$. Consider the strategy $(W_t)_{t\ge 1}$ (in $G$) defined by $W_t = \{v: v^{i(t)} \in W^C_t\}$. 
Observe that $|W_t|<h(G)$ implying the existence of a rabbit {escape walk} $(v_t)_{t \geq 1}$ allowing him to survive the hunter strategy. The rabbit escape {walk} in $G$ can be transformed into an escape {walk} $(u_t)_{t \geq 1}$ in $C_G^p$ where  $u_t = v^{i(t)}_t$.  This leads to contradiction since $(W^C_t)_{t\ge 1}$ was assumed to be a winning strategy. Hence, $h(C_G^p)=p \times h(G)$. 
    \end{proof}

Given two graphs $G=(V(G),E(G))$ and $H = (V(H),E(H))$, let  
$G \nabla H$ be  {the join} graph obtained by considering the union of $G$ and $H$ and fully connecting   $V(G)$ and $V(H)$: $V(G \nabla H) = V(H) \cup V(G)$ and $E (G \nabla H) = E(G) \cup E(H) \cup \{ uv: u \in V(G), v \in V(H)  \}$.  The next lemma states that $h(\cdot)$ is superadditive with respect to $\nabla$ and provides  an obvious upper bound. 

\begin{lemma}
$h(G) + h(H)  \leq h(G \nabla H) \leq \min \left(h(G) + n(H), h(H) + n(G) \right)$.
\label{lem:comp}
\end{lemma}

\begin{proof}
Consider a winning strategy $(W_t)_{t \geq 1}$ in $G$ using $h(G)$ hunters. Then $(V(H) \cup W_t)_{t \ge 1}$ is obviously a winning strategy in  $G \nabla H$ showing that  $h(G) + n(H)$ is an upper bound for $h(G \nabla H)$. By symmetry, $h(H) + n(G)$ is also an upper bound.  The lower bound is proved using exactly the same technique already used in the proof of Lemma \ref{lem:C}. More precisely, given any strategy
$(W_t)_{t \geq 1}$ using strictly less than $h(G)+h(H)$ hunters, we have either $|W_t \cap V(G)| <h(G)$ or  $|W_t \cap V(H)| <h(H)$ allowing to build a rabbit escape strategy.  
    \end{proof}

Observe that Lemma \ref{lem:comp} implies that $h(G \nabla {K}^{\circ}_k) = h(G) + k$ while two inequalities can be obtained if $H = K_k$:   
$h(G) + k -1 \leq h(G \nabla {K}_k) \leq h(G) + k$.
{Note that $h(K_k \nabla {K}_k)=h(K_{2k})=2k-1=h(K_k)+n(K_k)>h(K_k)+h(K_k)$ showing that the upper bound is sharp. 
On the other hand, let $G(a,p)$ be a graph obtained as the union of  {${K}^{\circ}_{a+1}$ and a path $v_1...v_p$  on $p \geq 2$ vertices starting at the clique (so $v_1 \in {K}^{\circ}_{a+1}$). Observe that $n(G(a,p)) = a+p$ and $h(G(a,p)) = a+1$.}   Now consider the join of two distinct copies of $G(a,p)$: we have $h(G(a,p) \nabla G(a,p))\le \max(2a+2,a+p+2)$. Indeed, two hunters can decontaminate the path on the first copy, while $a+p$ hunters {are shooting at} 
the second copy, then, for one step, $2(a+1)$ hunters can cover the clique
on both copies.
At this point the 
{rabbit's} territory is reduced to $\{v_2,...,v_p\}$ on the second copy. From there $a+p$ hunters {shoot at} 
the  first copy, while two hunters decontaminate the path in the second copy. When {$2 \leq p\le a$, we have $ h(G(a,p) \nabla G(a,p))=2a+2=h(G(a,p))+h(G(a,p))<h(G(a,p))+n(G(a,p))$, implying that the lower bound is also  sharp.}

\section{On the computational complexity of $h(G)$}
\label{sec:comp}

We start with a reduction from \textsc{3-partition} to the problem of computing $h_S(G)$. Then the latter is reduced to the problem of computing $h(G)$ in a graph having loops. Finally, using Lemma \ref{lem:bipartite}, we deduce that computing the hunting number is NP-hard in a bipartite graph.\\
Remember that an instance of \textsc{3-partition} is a multiset $S$ of $n$ positive integers $\{a_1,...,a_n\}$ with $n=3m$ for which we aim to decide whether there is a partition $S_1,...,S_m$ of $S$ such that the sum of the elements in each $S_j$ equals $\beta = 1/m \sum_{i=1}^n a_i$.
This problem is NP-hard even when the $a_i$ are bounded by a polynomial in $n$ and $\frac{\beta}{4} < a_i < \frac{\beta}{2}$ for $i=1,...,n$ \cite{gareyjohnson79}. Note that the latter condition implies that the $S_j$ are triplets.}

\begin{figure}[htbp]
    \centering
    \includegraphics[scale=0.45]{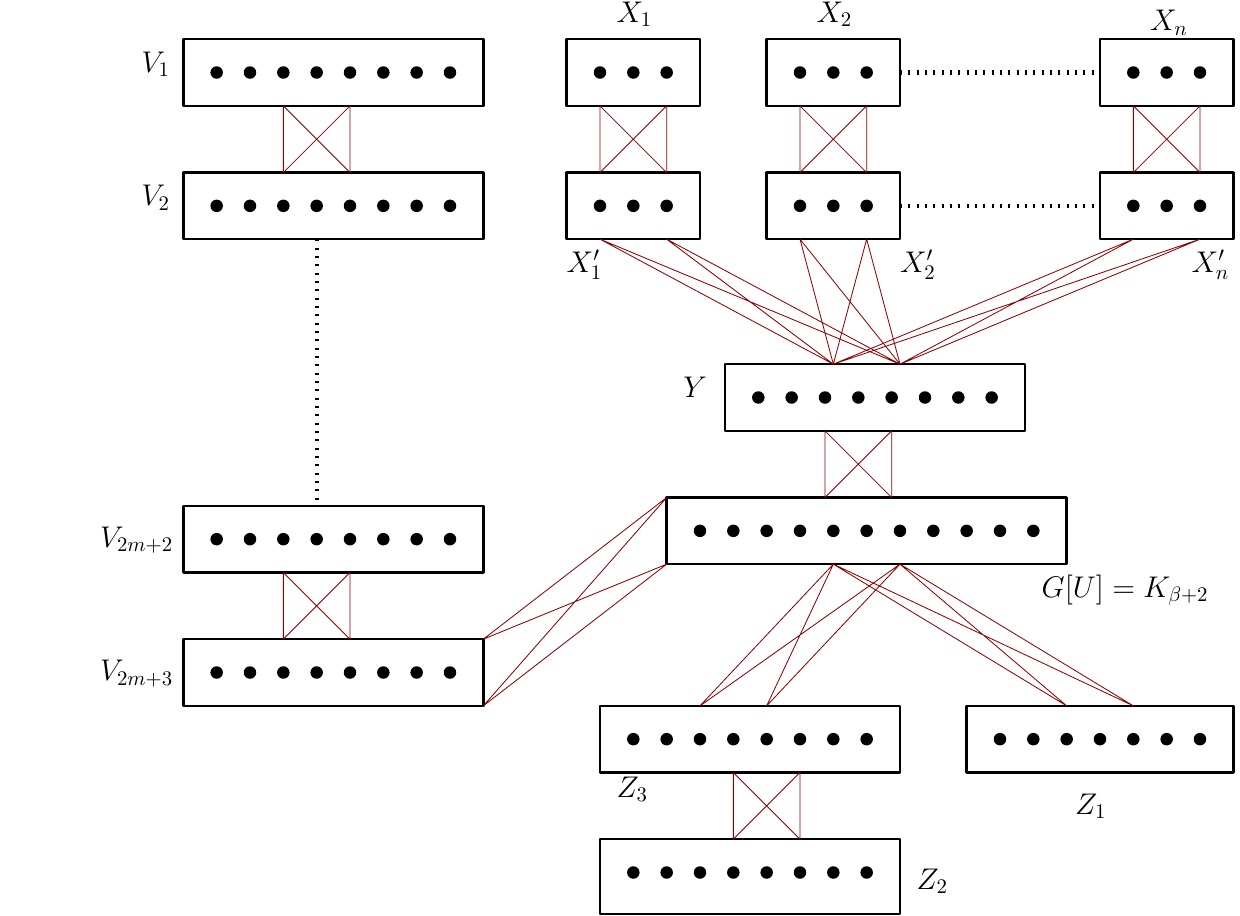}
    \caption{Illustration for the construction used in the Proof of Proposition~\ref{P:hard}. Here, each block of vertices other than $U$ is an independent set and $U$ is a clique. Further, $Y,Z_1,Z_2,Z_3, V_1,\ldots, V_{2m+3}$ contain $\beta$ vertices each, $U$ contains $\beta+2$ vertices, $X_i$ and $X'_i$ contain $a_i$ vertices (for $i\in [n]$). Finally, whenever two blocks, say $A$ and $B$, are illustrated to be connected by a red connection, 
    $A$ and $B$ are fully connected.}
    \label{fig:constrainedHardness}
\end{figure}

\begin{proposition}\label{P:hard}
It is  NP-hard to compute $h_S(G)$.
\end{proposition}
\begin{proof}
Let $\mathcal{S}$ be an instance of \textsc{3-partition}. We show that   $\mathcal{S}$ admits a \textsc{3-partition} if and only if $h_S(G) = \beta$ where $G$ and $S$ are described below and shown in Figure~\ref{fig:constrainedHardness}. The instance $\mathcal{S}$ defined by the numbers $(a_i)_{i \in [n]}$ is chosen as described above. 

\smallskip
\noindent\textbf{Construction of $G$.} Let $Y,Z_1,Z_2, Z_3,V_1,\ldots,V_{2m+3}$ each be a set of $\beta$ independent vertices. Further, let $U$ be a set of $\beta+2$ vertices that induces a clique in $G$. Finally, let $X_1,\ldots,X_n,X'_1,\ldots,X'_n$ be sets of independent vertices such that $|X_i|=|X'_i|=a_i$. 

Now, for $i\in [2m+2]$, fully connect $V_i$ to $V_{i+1}$ (i.e., $G[V_i\cup V_{i+1}]$ induces a complete bipartite graph with both partitions containing $\beta$ vertices). Similarly, fully connect $X_i$ to $X'_i$, and $Z_2$ to $Z_3$. Moreover, $Z_3 \cup Z_1 \cup V_{2m+3}$ and $U$ are fully connected, and for $i\in [n]$, each $X'_i$ is fully connected to $U$. Finally, let $S = Z_1\cup Z_2\cup V_1 \cup X_1 \cup \cdots \cup X_n$ be the allowed starting positions of the rabbit. 

In one direction, suppose $\mathcal{S}$ admits a \textsc{3-partition} $S_1,\ldots,S_m$. Then, we claim that the following hunter strategy is a winning one (here with a slight abuse of notation we are indicating with $S_j$ the set $\bigcup_{i | a_i \in S_j}X_i'$): $Z_1, Z_3, Y, S_1, Y,  S_2, Y, S_3, Y,\ldots,S_m, V_{2m+3}, V_{2m+2},\ldots, V_2$. 
To ease the exposition, we provide a case by case analysis distinguished by the starting positions of the rabbit.
\begin{enumerate}
    \item The rabbit starts in $Z_1$: The rabbit is {shot} in the first round since $W_1 = Z_1$.
    \item The rabbit starts in $Z_2$: 
    The rabbit is {shot}  in the second round.
    \item The rabbit starts in $X_i$ for some $i\in [n]$:  In this case, first we establish that the rabbit will be restricted to vertices in $X_1,\ldots,X_n, X'_1,\ldots,X'_n$ until round $2m+{2}$. To this end, observe that if the rabbit needs to leave these vertices, it needs to reach a vertex in $Y$, and it can only do so {at an odd time step greater than $1$. But this is not possible since $W_{t}=Y$ for odd rounds $3\le t \le 2m+1 $.\\
    Note that $G[X_1\cup \ldots\cup X_n\cup X'_1\cup \ldots\cup X'_n]$ is bipartite, thus $R_t \subseteq X_1\cup \ldots \cup X_n$ when $t$ is odd and $R_t \subseteq X_1'\cup \ldots\cup X_n'$ when $t$ is even, until round $2m+2$. 
    Furthermore, observe that when $W_t\supseteq X'_i$, the set $X_i\cup X'_i$ becomes decontaminated and remains so for all subsequent rounds (until round $2m+2$)} unless the rabbit moves to some vertex in $Y$, which is not possible.
   Finally, let the subset $S_j$, $j\in [m]$, contain numbers $a_p,a_q,$ and $a_r$. Then, in the round $2j+2$, hunters shoot at all vertices in $X'_p,X'_q,X'_r$, and hence decontaminate
    $X_p\cup X'_p\cup X_q\cup X'_q \cup X_r \cup X'_r$ { for all subsequent rounds (until round $2m+2$)}. Since $S_1,\ldots, S_m$ form a partition, after $2m+2$ rounds all vertices in $X_1,\ldots,X_n, X'_1,\ldots,X'_n$ will be decontaminated and as noted above, the rabbit is restricted to only these vertices for all these rounds. Hence, the rabbit gets shot. 
    
    \item The rabbit starts in $V_1$: {Observe that the induced graph $G[\bigcup_{i \in[2m+3]} V_{i}]$ is bipartite, therefore, since $R_1=V_1$ and $W_{2m+3} = V_{2m+3}$, we have $R_{2m+3} = \bigcup_{i \in[m+1]} V_{2i-1} $. }
    Then at time $2m+4$, we have $W_{2m+4} = V_{2m+2}$ and  $R_{2m+4} = \bigcup_{i \in[m]} V_{2i}$. A simple induction on $t$ shows that, for $1\le t\le m+1$, $R_{2m + 2t} = \bigcup_{i \in[m-t+2]} V_{2i}$ and
    $R_{2m + 2t + 1} = \bigcup_{i \in[m-t+2]} V_{2i-1}$. Therefore, {$R_{4m+3} =V_1$,} $W_{4m+4} = V_2$ and thus $R_{4m+4} = \emptyset$.
\end{enumerate}

In the other direction, suppose $\mathcal{S}$ does not admit a \textsc{3-partition}. We will show that the hunters do not have a winning strategy using only $\beta$ hunters. Assume, by contradiction, that such a winning strategy exists. We begin by observing that once the rabbit reaches a vertex in $U$ {(i.e., $R_{t}\cap U \neq \emptyset$ for some $t>0$)}, then the rabbit can never be shot since the hunting number of $G[U]$ is greater than $\beta$. {Thus, to complete our proof, we only need to show that for any strategy of $k$ hunters the rabbit has a walk that ensures $R_{t}\cap U \neq \emptyset$ for some $t>0$. Furthermore, {observe that} all $\beta+2$ vertices of $U$ are twins and there are at most $\beta$ hunters, {thus} if $R_{t-1}\cap (Y \cup Z_3 \cup Z_1 \cup V_{2m+3}) \neq \emptyset$ for some $t>1$, then $R_{t}\cap U \neq \emptyset$.}
As  a consequence, we can safely assume that $W_1=Z_1$ and $W_2={Z_3}$ (otherwise the rabbit can reach $U$ in time step 2 and 3, respectively). Similarly, $W_3 = Y$, otherwise the rabbit has an escape strategy by starting on some $X_i$, moving to $X_i'$ in the second time step, then to some $v \in Y \setminus W_3$ in the third time step, and finally moving to $U\setminus W_4$ in the fourth time step.\\  
Let $T$ be the first time step such that $R_T \cap \bigcup_{i \in [n]} (X_i \cup X'_i) = \emptyset $.  For each odd time step $t$ between $3$ and $T$, $W_t = Y$ holds (otherwise the rabbit can reach $U$ in the next time step). Observe that $T$ is necessarily an even number. {Moreover, observe that since hunters are shooting at $Y$ in every odd time step, to decontaminate $X_i\cup X'_i$ the hunters must shoot at every vertex in $X'_i$ in some even round.} Furthermore, since hunters were able to decontaminate $\bigcup_{i \in [n]} (X_i \cup X'_i)$, they need to shoot at vertices in $\bigcup_{i \in [n]}  X'_i$ for at least $m+1$ (even) time steps. This is obviously due to the non-feasibility of the \textsc{3-partition} instance implying that it is not possible to cover all vertices of $\bigcup_{i \in [n]}  X'_i$ with only $m$ subsets of size $\beta$. For $ t \leq T$, let $l_t$ be the largest index such that $R_t \cap V_{l_t} \neq \emptyset$ (i.e., the index of the lowest $V_i$ in Figure \ref{fig:constrainedHardness} that is not decontaminated {at time $t$}). We clearly have $l_1 = 1$, $l_2 = 2$ and $l_3 = 3$.  Observe that if at some even time $t$ the hunters shoot at $V_{l_{t-1}+1}$, then $l_t = l_{t-1} - 1$. 
Otherwise, if the hunters shoot at other vertices ({i.e., $W_t \neq V_{l_{t-1}+1}$}), for example, in  $\bigcup_{i \in [n]}  X'_i$, then $l_t = l_{t-1} + 1$. If $t$ is odd,  $W_t = Y$ implying that $l_t = l_{t-1} + 1$.  Then, if we consider two consecutive time slots $t$ and $t+1$, we either have $l_{t+1} = l_{t-1}$ or $l_{t+1} = l_{t-1} + 2$.  
Since hunters have to shoot at $\bigcup_{i \in [n]}  X'_i$ for at least $m+1$  time steps, $l_{t+1}-l_{t-1}$ increases by $2$, at least $m+1$ times (for some odd $t$). Since $T$ is even and  hunters are shooting at (even partially)  $\bigcup_{i \in [n]}  X'_i$ at time $T$,  $l_{t+1}-l_{t-1}$ increases by $2$, at least $m$ times between $3$ and $T-2$. One can then write that {$l_{T-2} -l_2 = (l_{T-2} - l_{T-4}) + (l_{T-4}-l_{T-6})+\cdots + (l_{4} - l_2) \geq 2m$} implying that $l_{T-2} \geq 2m+2$ and $l_{T-1} \geq 2m+3$ (since $W_{T-1} = Y$). The rabbit can then reach $U$ at time $T$ implying that the hunter's strategy was not a winning one.  \end{proof}

\begin{figure}[htbp]
\centering
\vspace{-10mm}
\includegraphics
[scale=0.25]{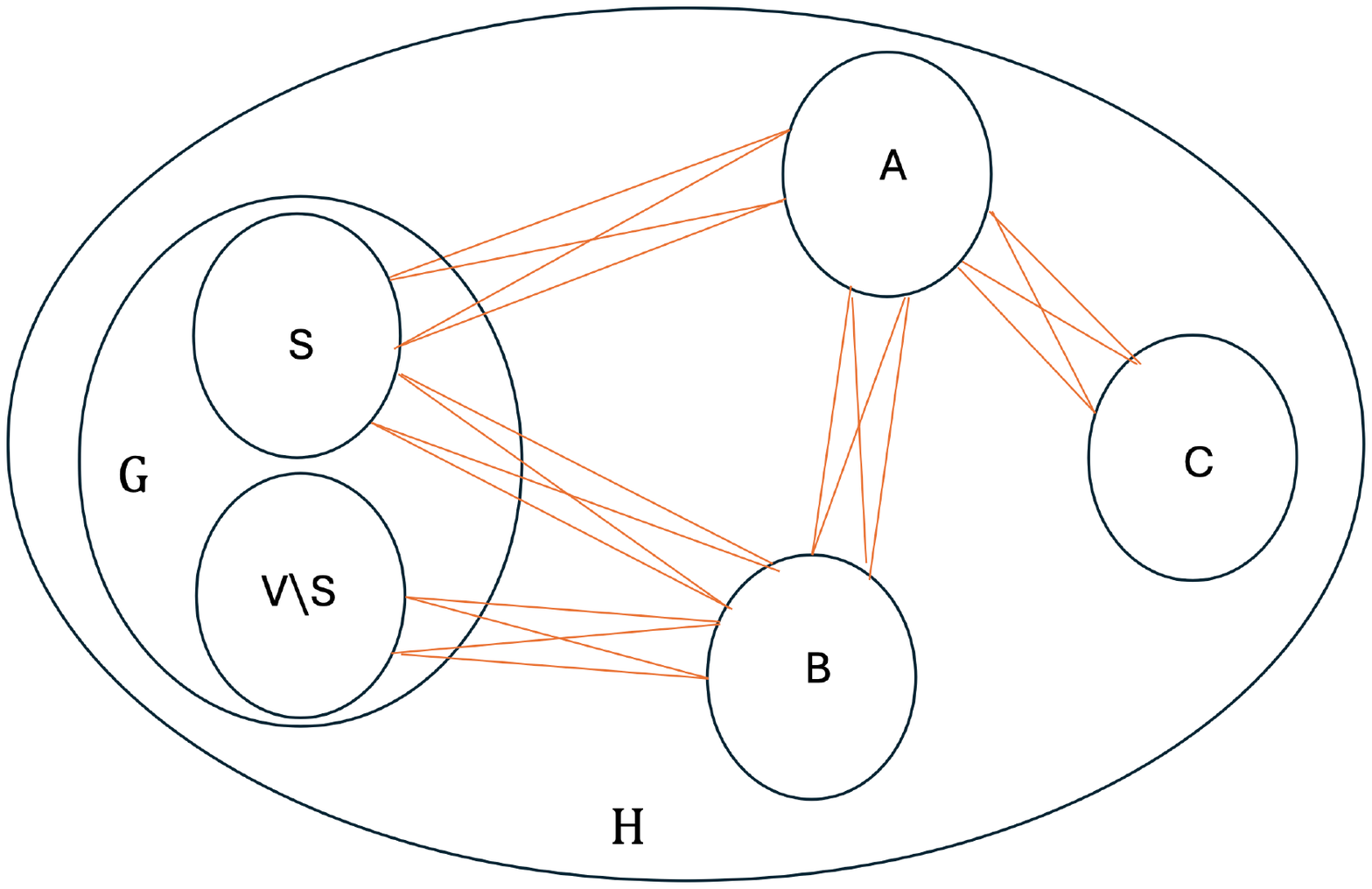} 
    \caption{Building $H$ from $G$ (Proposition \ref{pro:loop}): $1 \leq k \leq |S|$, $|A|=n-k$, $|B|=k$, $|C|=2k$, $H[A]={K}^{\circ}_{n-k}$, $H[B]={K}^{\circ}_{k}$ and $H[C]={K}^{\circ}_{2k}$ are complete graphs with loops.}
\label{fig:reduc2}
    \end{figure}

\begin{proposition}
It is NP-hard to compute $h(G)$ in a graph with loops.
\label{pro:loop}
\end{proposition}
\begin{proof}
We prove the result using a reduction from the problem of computing $h_S(G)$.  
Consider a graph $G=(V,E)$ and a subset  $S \subset V$ representing the possible initial positions of the rabbit. Let $n=|V|$ and  let $k$ be a number satisfying $1 \leq k \leq |S|$. We build a graph $H$ as follows. The graph $H$ contains $G$ as a subgraph in addition to $3$ complete subgraphs with loops: {$H[A]={K}^{\circ}_{n-k}$, $H[B]={K}^{\circ}_{k}$ and $H[C]={K}^{\circ}_{2k}$}. The set $A$ is fully connected to $B$, $C$ and $S$  while $B$ is fully connected to $V$ (and $A$) (see Figure \ref{fig:reduc2}). 
Observe that {$H[A \cup C]={K}^{\circ}_{n+k}$}  implying that $h(H) \geq n + k$. We claim that $h(H) = n+k$ if and only if $h_S(G) \leq k$. 
\begin{claim*}
$h_S(G) \leq k$, if and only if $h(H) = n+k$.
\end{claim*}
\begin{proof}
Assume that $h(H) = n +k$ and $h_S(G) > k \geq 1$. Let $W_1,...,W_T$ be a winning strategy in $H$ using $n+k$ hunters.\\ 
Let us first assume that there exists $t \geq 1$ such that $R_t \cap (A \cup B)  \neq \emptyset$.  Let $\delta$ be the last $t$ such that $R_t \cap (A \cup B)  \neq \emptyset$. 
Then, $W_{\delta + 1} \supset (A \cup B) $ holds, {implying $|W_{\delta + 1} \cap (V\cup C)| \leq k$}. \\
If $R_{\delta} \cap A \neq \emptyset$, then $R_{\delta+1} \cap C \neq \emptyset$ and $R_{\delta+1} \cap S \neq \emptyset$ since $|S| >k$ and $|C| >k$.  Using that $S$ (and even $V$) is connected to $B$ and $C$ is connected to $A$, we deduce that $W_{\delta + 2}$ should also contain $ A \cup B$ {and $R_{\delta+2} \cap C \neq \emptyset$}. By induction on $t \geq \delta + 1$, as long as $R_t \cap V \neq \emptyset$,  we  have $W_{t+1} \supset A \cup B$ and $R_{t+1} \cap C \neq \emptyset$.
Since the game is supposed to end,  
there exists some $T'$ between $\delta+1$ and $T$ such that $R_{T'} \cap V = \emptyset$. {Observe that between rounds $\delta+1$ and $T'$ at most $k$ hunters can shoot at vertices in $V$, while $R_{\delta}$ contains a vertex (of $A$) fully connected with $S$. Therefore, by considering the game played between rounds $\delta+1$ and $T'$ restricted to $V$, we obtain $h_S(G)\le k$, which is a contradiction.\\ }
Assume now that $R_{\delta} \cap B \neq \emptyset$, then $R_{\delta + 1} \cap S \neq \emptyset$ implying that $W_{\delta + 2} \supset (A \cup B) $ (therefore,  $|W_{\delta + 2} \cap V| \leq k$). Since the rabbit {can enter} $V$ (from $B$) at time $1 + \delta$ and $k <h_S(G)$, we can deduce from Lemma \ref{nodeconaminationS} that $R_{\delta + 2} \cap S \neq \emptyset$.
{A simple} 
 induction on $t \geq \delta + 1$ {shows that we will always  have $W_{t+1} \supset A \cup B$ and $R_{t+1} \cap S \neq \emptyset$ implying that} the game will never end. \\
 {Note that a crucial fact in the proof of the first (resp. second) case is that the rabbit {can enter} $S$ (resp. $V$) at time $1+\delta$, and at most $k <h_S(G)$ hunters can {shoot at vertices of} $V$ from that time on.\\} Let us now assume that $\delta $ does not exist which is equivalent to say that $R_t \cap (A \cup B) = \emptyset$ $\forall t \geq 1$. We consequently have $W_1 \supset (A \cup B)$, $R_1 \cap V \neq \emptyset$ and $R_1 \cap {S \neq \emptyset}$. 
 The situation is then similar to the previous case where we had $R_{\delta} \cap B \neq \emptyset$ since the rabbit can be at any vertex of $V {\setminus W_1}$ at time $1$. The same induction on $t$ shows that $W_{t+1} \supset A \cup B$ and $R_{t+1} \cap S \neq \emptyset$ implying that the game is endless. \\
Finally, let us now show that $h(H) = n+k$ when $k \geq h_S(G)$. This  can be done by providing a winning strategy. A possible one starts with $W_1 = B \cup V$ leading to $R_1 = A \cup C$.  Assume that $W'_1, W'_2,..., W'_p$ is a winning strategy (for some time $p$) allowing to {shoot} the rabbit in $G$ if he starts at $S$. We can then take $W_2 = W'_1 \cup A \cup B$, $W_3 = W'_2 \cup A \cup B$, etc. So $W_{p+1} =W'_p \cup A \cup B$. At time $p+1$ we have $R_{p+1} = C$. Therefore, by setting   $W_{p+2} = A  \cup C$, the whole graph is decontaminated. 
    \end{proof}
The claim immediately proves the result since one can compute $h_S(G)$ by  computing the {hunting number} of $O(log(n))$ graphs of type $H$ for different values of $k$. 
    \end{proof}

We are now able to state the main complexity result. 

\begin{theorem}
It is NP-hard to compute $h(G)$ in a  bipartite graph.
\label{th:main}
\end{theorem}
\begin{proof}
This is a consequence of Lemma \ref{lem:bipartite} and Proposition  \ref{pro:loop}.
    \end{proof}


Let us focus now on the existence of polynomial-time approximation algorithms with additive guarantees. 
An obvious $O(n)$ additive guarantee is given by the upper bound $n$ for the {hunting number}.  We prove that  it is not possible to do much better than $O(n)$.  
\begin{theorem}
It is NP-hard to additively approximate $h(G)$ within $O(n^{1 - \epsilon})$ for any constant $\epsilon >0$. \label{th:add}
\end{theorem}
\begin{proof}
Assume that we have a polynomial-time approximation algorithm with an $O(n^{1 - \epsilon})$ additive guarantee. Given any graph $G$, let us build the graph $C_G^p$ with $p= \lceil n^{2\frac{1-\epsilon}{\epsilon}}\rceil $. Using the approximation algorithm, we get an upper bound $u$ satisfying inequalities $h(C^p_G) \leq u \leq h(C^p_G) + O({(np)}^{1-\epsilon})$. From  Lemma \ref{lem:C}, we know that $h(C^p_G) = p \times h(G)$ implying that $ \frac{u}{p} -\frac{1}{p} O({(np)}^{1-\epsilon}) \leq h(G) \leq \frac{u}{p}$.  Using that $p= \lceil n^{2\frac{1-\epsilon}{\epsilon}} \rceil$ leads to $\frac{u}{p} - O(\frac{1}{n^{1-\epsilon}}) \leq   h(G)  \leq  \frac{u}{p}$. Since $h(G)$ is integer and $O(\frac{1}{n^{1-\epsilon}})$ is negligible,
we get a polynomial-time algorithm to compute $h(G)$ (the construction is obviously polynomial for constant $\epsilon$). Theorem \ref{th:main} allows to conclude.
\end{proof}

Since graphs for which $h(G) = 1$ or $h(G)= n$ are well characterized, one might expect polynomial-time algorithms for either small values of $h$ or large values of $h$. The next theorem states that this is not the case {when {either} $h$ {or} $n-h$ are upper bounded by a (small) power of $n$}.  

\begin{theorem}
It is NP-hard to compute $h(G)$ for simple graph instances where  $n - h(G) = O(n^{\epsilon})$ (resp. $h(G) = O(n^{\epsilon})$) for any constant $\epsilon > 0$.
\end{theorem}
\begin{proof}
Given a simple graph $G$ with $n = |V(G)|$, take $k$ so that $n \le k^{\epsilon}$ and  consider the graph $G \nabla {K}_k$.  Notice that $G \nabla {K}_k$ does not contain loops.  
From Lemma \ref{lem:comp}, we have $h(G)+k-1 \leq h(G \nabla {K}_k)  \leq  h(G) +k$ implying that {$h(G \nabla {K}_k) -k \leq  h(G) \leq  h(G \nabla {K}_k) -k+1$.} Furthermore $n(G \nabla {K}_k) - h(G \nabla {K}_k) \leq  (n+k)- h(G) -k+1\le n \le (n+k)^{\epsilon}$, therefore  $n(G \nabla {K}_k) - h(G \nabla {K}_k)$ is $O(n(G \nabla {K}_k)^{\epsilon})$.   A polynomial-time algorithm  computing the {hunting number} if  {$n-h=O(n^{\epsilon})$, applied on $G \nabla {K}_k$,} can  be used to get a lower bound and an upper bound for $h(G)$, whose difference   is $1$. In other words, we have a constant additive  approximation for $h(G)$ contradicting Theorem \ref{th:add}. \\
To prove NP-hardness for instances where $h = O(n^{\epsilon})$, one can start from a graph $G$ and add a stable set of size $\Omega(n^{\frac{1}{\epsilon}})$ to get a new graph $H$ for which $h(G) = h(H) = O(n(H)^{\epsilon})$.    \end{proof}

\section{Hunting the rabbit within some time limit}

A natural extension of the game consists in imposing a time limit $l$, i.e., the rabbit should be shot no later than time $l$. 
For $l=1$, we obviously need $n$ hunters, while $h(G)$ hunters are needed when there is no time limit. More generally, let $h(G,l)$ denote the minimum number of hunters needed to guarantee the capture of rabbit no later than $l$. Clearly, $h(G,l)$ is a non-increasing function of $l$. {Recently, it was established that to obtain $h(G) = h(G,l)$, $l$ may required to be super-polynomial in $n(G)$~\cite{althoetmar2025}. We begin with the following observation.}
\begin{observation}\label{polyconstantkl}
When $l$ and $k$ are constant, a brute force search provides a polynomial algorithm to determine whether $h(G,l)\le k$.
\end{observation}
   In contrast, we will prove that computing $h(G,l)$ is NP-hard for every $l \geq 2 $ that can be upper bounded by a polynomial in the order of $G$. The main steps of the proof are the following. 
We first show that it is NP-hard to compute $h(G,2)$ based on some connections with separator problems. Then we prove that $h(G,3)$ is also NP-hard to compute. Afterwards, we provide a reduction from the computation of  $h(G,l)$ to the computation of $h(G,l+2)$, thus allowing us to prove the main theorem.

Let us now recall some well-known results related to vertex separators.
An $\alpha$-vertex separator $C$ is generally defined as a subset of vertices such that $V(G) \setminus C = A \cup B$ where $|A| \leq \alpha n$,  $|B| \leq \alpha n$ and there are no edges between $A$ and $B$. We know from \cite{BUI} that computing a minimum size separator {is} NP-hard even for $\alpha = 1/2$.   Notice that $A$ and $B$ generally have different sizes.
A more restrictive version called vertex bisection is studied in \cite{Balance}  where we require that $| |B|-|A| | \leq 1$. The vertex bisection problem is also NP-hard.  
 Let us  consider the slightly 
modified problem, entitled \emph{exact vertex bisection problem}, where we look for a minimum size separator $C$ separating $V \setminus C$ into two equal size disjoint subsets $A$ and $B$, described as follows. 
{\small
\begin{equation*}
(\mbox{exact vertex bisection problem}): 
\left\{\begin{array}{ll}
&\min |C|   \\
& (A, B, C) \mbox{ is a partition of } V(G) \\ &|A|=|B| \\
& N(A) \subseteq A \cup C,  N(B) \subseteq B \cup C  \\
\end{array}
\right.
\end{equation*}
}
{We will denote by $EVB(G)$ the minimum value of the above problem (i.e., $|C|$).}

\begin{lemma}
The exact vertex bisection problem is NP-hard.
\label{lem:exact} 
\end{lemma}
\begin{proof}
We already mentioned that when we only require $| |B|-|A| | \leq 1$, computing a minimum size separator is known to be NP-hard \cite{Balance}. Moreover, it is easy to reduce the slightly relaxed  problem to  the one with exact equality. We only have to solve two instances with equality constraints: we solve one problem on the original graph $G$  where we require that $|A|=|B|$, and we solve a second one by adding one isolated  vertex to $G$.  In the second case, there exists an optimal separator $C$ not containing the newly added vertex, and the obtained subsets $A$ and $B$ in the new graph have the same size and will then induce, in the original graph,  two sets whose sizes differ by $1$. An optimal solution of the vertex bisection problem is just obtained by considering the best separator among the two obtained above.
\end{proof}

Consider the case where $l=2$. It is easy to see that $h(G,2)$ is given by  $\min_{S \subset V} \max(|S|,|N(\overline{S})|)$ (where $\overline{S}= V \setminus S$). Indeed, $S$ can be seen as $W_1$ while $N(\overline{S})$ is $W_2$.
We can also write that $h(G,2) = \min_{S \subset V} \max(|\overline{S}|,|N({S})|)$.  

  To handle the  case $l=2$, we need to introduce another separator problem that might be entitled \emph{bisection stable separator problem}. We look for a subset $C$ separating $V(G) \setminus C$ into $3$ subsets: $A$, $B$ and $D$ such that $|A|=|B|$, $D$ is a stable set, and there are no edges joining any two sets among $A$, $B$ and $D$. We aim to minimize $|C|-|D|$. The problem can formally be described as follows. 
{\small
\begin{equation*}
(\mbox{bisection stable separator problem}): 
\left\{ \begin{array}{ll}
&\min |C| - |D|  \\
& (A, B, C, D) \mbox{ is a partition of } V(G)
\\
& |A|=|B|,  N(D) \subseteq C \\
& N(B) \subseteq B \cup C,  N(A) \subseteq A \cup C  \\
\end{array}
\right.
\end{equation*}
}
{We will denote by $BSS(G)$ the minimum value of the above problem (i.e., $|C|-|D|$).}

We show below the link between $h(G,2)$ and the bisection stable separator problem.

\begin{lemma} Given any graph $G = (V,E)$, $
h(G,2) = \frac{1}{2} \Big( n + BSS(G) \Big)
$.
\label{lem:h2g}
\end{lemma}
\begin{proof}
Consider a winning strategy $W_1, W_2$ in $G$ such that $\max(|W_1|,|W_2|)= h(G,2)$. We can assume that $|W_1|= |W_2|$.
Let $C = W_1 \cap W_2$, $B = W_2 \setminus W_1$, $A = W_1 \setminus  W_2$ and $D = \overline{W_1 \cup W_2}$. From $|W_1|=|W_2|$ we get that $|A|=|B|$. $D$ is obviously a stable set, because otherwise the rabbit would not be captured using the strategy $W_1, W_2$. Moreover, there are no edges between   $B$ and $A$, nor between $D$ and $A$, since a rabbit starting at $B \cup D = \overline{W_1}$ can survive by moving  to $A$ at the second time step.  Similarly, there are no edges between $B$ and $D$,  since a rabbit starting at $B$ can survive by moving  to $D$ at the second time step. $(A,B,C,D)$ is then a feasible bisection stable separator. 
Furthermore,   $n = |C|+ 2|A|+|D|$ implying that $h(G,2)=|W_1|= |C|+ |A|= n/2 + (|C|-|D|)/2$. \\
Conversely, an optimal bisection stable separator defined by a partition $(A,B,C,D)$ can obviously be used to construct a winning strategy by taking $W_1 = C \cup A$ and $W_2 = C \cup B$, ending the proof. \end{proof}

To conclude that $h(G,2)$ is difficult to compute, we need to show that the bisection stable separator problem is NP-hard. 
{To do that, we will reduce exact vertex bisection  to the bisection stable separator problem.  While the reduction in Lemma \ref{lem:exact} might produce an isolated vertex, the construction in next lemma does not give rise to a graph with isolated vertices. In other words, the bisection separator problem is NP-hard even for simple graphs not containing isolated vertices.} 

\begin{lemma}
The bisection stable separator problem is NP-hard.
\label{lem:bisec}
\end{lemma}
\begin{proof}
We prove NP-hardness using a reduction from the exact vertex bisection problem considered in Lemma \ref{lem:exact}. Consider an instance of the exact vertex bisection problem  given by a graph $G=(V,E)$.  Assume that $(A, B,C)$ is an optimal exact vertex bisection and let $EVB(G) = |C|$. \\ 
Let $M$ be a large number (for example, $M = 3n$) and consider the graph $G'=(V',E')$ obtained from $G$ by replacing each vertex $v \in V$ by a clique  $H_v$ of size $M$, and adding edges between all vertices of $H_v$ and $H_w$ whenever $vw \in E$.  For convenience, we overload the notation $H_v$ to represent both the clique $H_v$ and the vertices of $H_v$. 
Assume that $(A', B',C',D')$ is an optimal solution of the bisection stable separator problem related to $G'$ and let $BSS(G')= |C'|-|D'|$.  We prove that $EVB(G) \leq k$, if and only if, $BSS(G') \leq k M $.\\
Assume that $EVB(G) \leq k$ and consider the related partition $(A, B,C)$. One can then build a feasible solution of the bisection stable separator problem by setting $A'= \bigcup_{v \in A} H_v$, $B'= \bigcup_{v \in B} H_v$, $C'= \bigcup_{v \in C} H_v$ and $D'=\emptyset$. Therefore, $BSS(G') \leq k M $.  \\
Conversely, assume that $BSS(G') \leq  k M$ and let $(A', B', C',D')$ a corresponding optimal partition.  Since $D'$ is a stable set, $|D'| \leq n$ holds  (at most one vertex from any $H_v$ can be in $D'$).  Observe that if $H_v \cap B' \neq \emptyset$, then $H_v \cap A' = H_v \cap D' = \emptyset$. 
Assume that there are two vertices $v, w \in V$ such that $H_v \cap B' \neq \emptyset$, $H_v \cap C' \neq \emptyset$,  $H_w \cap A' \neq \emptyset$ and $H_w \cap C' \neq \emptyset$. Then one can move one vertex from $H_v \cap C'$ to $B'$ and one vertex from  $H_w \cap C'$ to $A'$ leading to a better bisection stable separator (since $|C'|$ is reduced while $|D'|$ did not change). Without loss of generality, we can then assume that if $H_v \cap A' \neq \emptyset$, then the whole clique $H_v$ is inside $A'$. We consequently have $|A'| = M r$ where $r$ is some integer number.  
Assume that the number of vertices that are partially in $B'$ is greater than or equal to $2$:  $|\{v \in V: H_v \cap B' \neq \emptyset  \mbox{ and }  H_v \cap C' \neq \emptyset \}| \geq 2$.   Let then  $v, w \in V$ be such that $H_v \cap B' \neq \emptyset$,  $H_v \cap C' \neq \emptyset$, $H_w \cap B' \neq \emptyset$ and $H_w \cap C' \neq \emptyset$, and assume{, without loss of generality,} that $|H_w \cap C'| \leq |H_v \cap B'|$, then one can move all vertices of $H_w \cap C'$ to $B'$ and move $|H_w \cap C'|$ vertices from  
$H_v \cap B'$ to $C'$. Observe that $|A'|$, $|B'|$, $|C'|$ and $|D'|$ did not change, while the number of vertices that are partially in $B'$ is reduced  (since $H_w$ is now fully in $B'$).
We can therefore assume that $|\{v \in V: H_v \cap B' \neq \emptyset  \mbox{ and }  H_v \cap C' \neq \emptyset \}| \leq 1$.  Assume that $|\{v \in V: H_v \cap B' \neq \emptyset  \mbox{ and }  H_v \cap C' \neq \emptyset \}| = 1$. Then $B'$ contains a certain number of cliques of size $M$ with another partial clique ($H_v \cap B'$ for the only vertex for which  $H_v \cap B' \neq \emptyset$ and   $H_v \cap C' \neq \emptyset$). In other words, $|B'| =M s + t $ where $s$ is an integer number and $1 \leq t \leq M-1$. This is obviously not possible since $|B'|=|A'|= M r$.   As a consequence, $|\{v \in V: H_v \cap B' \neq \emptyset  \mbox{ and }  H_v \cap C' \neq \emptyset \}|=0$  implying that $B'$ contains only full cliques of size $M$.   Let us then define $B$ and $A$ as follows: $B = \{v \in V: H_v \subset B'\}$ and $A = \{v \in V: H_v \subset A'\}$. Observe that $|A|= |B|$ and there are no edges between $A$ and $B$ in $G$. We can also define $C = V \setminus (A \cup B)$.  Since $BSS(G') = |C'|-|D'| \leq k M  $, $|D'| \leq n$  and $|C'| \geq M |C| - n$, we get that $M |C| - 2n \leq  k M$ leading to $|C| \leq k$. Since $(A, B, C)$ is an exact vertex bisection, we deduce that $EVB(G) \leq k$. \\ The equivalence between $EVB(G) \leq k$ and $BSS(G') \leq kM$ is then proved. Since $G'$ is of polynomial size (for $M = 3n$), Lemma \ref{lem:exact} implies that the bisection stable separator problem is NP-hard. 
\end{proof}

\begin{proposition}
It NP-hard to compute $h(G,2)$ for simple graphs.
\label{prop:hG2}
\end{proposition}
\begin{proof}
This is a direct consequence of Lemmas \ref{lem:bisec} and \ref{lem:h2g}.
\end{proof}

Additional bounds for $h(G,2)$ can also be obtained.  
Let $M(G)$ denote the matching number of $G$, that is, the maximum size of a matching in $G$.  
We write $VC(G)$ for the minimum size of a vertex cover of $G$.  
For any graph $G$, we will also make use of the bipartite graph $B_G$, defined earlier in Section~\ref{sec:prel}.

\begin{proposition}\label{matchingvc} The following inequalities hold:
$$\lceil \frac{VC(G)}{2} \rceil \leq M(G) \leq \lceil \frac{VC(B_G)}{2} \rceil \leq h(G,2) \leq VC(G) \leq 2 M(G) \leq VC(B_G).$$
\label{pro:ineq}
\end{proposition}
\begin{proof}
If $S$ is a vertex cover of $G$ then we can apply the strategy $W_1 = W_2 = S$ to catch the rabbit showing that $h(G,2) \leq VC(G)$.  
Moreover, $VC(G) \leq 2M(G)$ is a standard result (the extremities of edges of a maximum matching define a vertex cover). 
Consider the bipartite graph $B_G$ introduced in Section \ref{sec:prel}. An edge $uv$ of $G$ leads to edges $uv'$ and $u'v$ of $B_G$. As a consequence, a matching in $G$ of size $M(G)$ will lead to a matching of size $2 M(G)$ in $B_G$ implying that $2 M(G) \leq M(B_G)$. Since $B_G$ is bipartite, $VC(B_G) = M(B_G)$ showing the rightmost inequality. 

Let $W_1, W_2$ be a winning strategy with $\max(|W_1|,|W_2|)=h(G,2)$. 
Since  $V(B_G) = V \cup V'$ where $V$ is the set of vertices of $G$ and $V'$ is a copy of $V$, we might project $W_1$ (resp. $W_2$) on $V$ (resp. $V'$) to get  $S_1 = \{v \in V: v \in W_1\}$ and $S_2 = \{v' \in V': v \in W_2\}$ (remember that $v'$ is just a copy of $v$).  Then, $S_1 \cup S_2$ is a vertex cover of $B_G$. Therefore,  $VC(B_G) \leq |S_1| + |S_2| \leq  2 h(G,2)$.  Finally, the two leftmost inequalities are exactly the same as the two rightmost inequalities.
\end{proof}
As an example, consider the case where $G = K_n$. Then $h(G,2) = n - 1$, $VC(G) = n -1$, $M(G) = \lfloor n/2 \rfloor$ and $VC(B_G) = n$ showing that the upper bounds here are almost tight.   On the other hand, if $G$ is the union of $2$ disjoint cliques (each of size $n/2$), then $h(G,2)=n/2$, $VC(G)=n - 2$, $M(G) = 2 \lfloor n/4 \rfloor$ and $VC(B_G) = n$ leading to tight lower bounds.

An immediate consequence of Proposition \ref{pro:ineq} is that $h(G,2)$ can be approximated within a $2$ multiplicative ratio. A general $l$-factor approximation will be presented at the end of this section. Using König's theorem, we can also deduce that $h(G,2)$ is  easy to compute when $G$ is bipartite.

\begin{corollary}
If $G$ is bipartite, then $h(G,2) = M(G) = VC(G)$  and can then be computed in polynomial time. 
\end{corollary}
 
 We will now prove that computing $h(G,l)$ is NP-hard for $l = 3$.  Let us first recall this symmetry result from \cite{benameur2024complexityresultscopsrobber}.

\begin{lemma}
 If $W_1, W_2, ..., W_l$ is a winning strategy, then $W_{l}, W_{l-1}, ..., W_1$ is also a winning strategy.
 \label{lem:sym}
\end{lemma}
\begin{proof}
Suppose,  for the sake of contradiction that,  $W_{l}, W_{l-1}, ..., W_1$ is not a winning strategy. Then there exists a rabbit escape strategy $r_1, ..., r_l$. Observe that the reverse walk  $r_l, r_{l-1}, ..., r_1$ allows the rabbit to escape against $W_1, ..., W_l$, thus contradicting the winning nature of the strategy.
\end{proof} 
 
 \begin{proposition}\label{2to3}
It is NP-hard to compute $h(G,3)$ for simple graphs $G$.
 \end{proposition}
 \begin{proof}
 We will reduce from the problem of deciding whether $h(G,2)\le k$ for given $G,k$. By Observation \ref{polyconstantkl}, we can assume in the whole proof that $k \geq 4$. Consider the graph $G'$ obtained as the disjoint union of $G$ and a clique $K_k$ on $k$ vertices. We claim that $h(G,2)\le k$ if and only if $h(G',3)\le k$.\\
 On one hand if $W_1,W_2$ is a winning strategy on $G$ with at most $k$ hunters, then $W_1,W_2,V(K_k)$ is a winning strategy on $G'$ with $k$ hunters.
 \\ Now suppose $W'_1,W'_2,W'_3$ is a winning strategy on $G'$ with at most $k$ hunters that minimizes the total number of shots $|W'_1|+|W'_2|+|W'_3|$ and let $C_i=W'_i\cap V(K_k)$ for $i=1,2,3$.\\ \ \\
{\textbf{Claim.}}
 Only three cases are possible. {\bf Case 1}: there exists a $C_i$ of size $k$; {\bf Case 2}: there are two consecutive $C_i$ of size $k-1$ and the remaining $C_j$ is empty; {\bf Case 3}: $|C_1|=|C_3|=k-1$ and $|C_2|=k-2$.\\ \ \\
{\textit{Proof.}} We will assume that no $C_i$ has size $k$ and derive the other two cases. First observe that if at some step $j$ the rabbit territory within the clique contains two distinct vertices, then all shots on the clique up to time step $j$ are useless, since the two vertices dominate the entire clique. In particular, given the minimality (in terms of total shots) of the hunter strategy, $C_1$ is either empty or contains exactly $k-1$ vertices:
 if $C_1$ is empty, then, for the same reason, $C_2$ can either be empty or contain exactly $k-1$ vertices, but if $C_2=\emptyset$, then $|C_3|=k$, which we assumed is not the case, so necessarily $|C_2|=k-1$ and therefore, $|C_3|=k-1$ (since $N(R_2)\cap K_k$ would have $k-1$ vertices). Therefore, if $C_1$ is empty we are in Case 2.\\
 If $|C_1|=k-1$, then either $|C_2|=k-1$ and, by minimality $C_3=\emptyset$ and we are in Case 2, or $|C_2|<k-1$, but the only possibility is $|C_2|=k-2$, otherwise $|R_2\cap K_k| \ge 2$
and the initial shots are wasted. Thus $|C_1|=k-1$ and $|C_2|=k-2$, implying $R_2$ has exactly one vertex in the clique and thus $|C_3|=k-1$ and we are in Case 3. \qed \\ \ \\
{\bf Case 1}: if $|C_1|=k$, then $W'_2\cap V(G), W'_3\cap V(G)$ is a winning strategy on $G$ with at most $k$ hunters, symmetrically, if $|C_3|=k$, then $W'_1\cap V(G), W'_2\cap V(G)$ is a winning strategy on $G$ with at most $k$ hunters. Assume $|C_2|=k$, then $W_2\cap V(G) =\emptyset$ implying $R_2\cap V(G) =N(R_1)\cap V(G)$  and thus $R_1\cap V(G)\subseteq N(R_2)\cap V(G)$, from which we have $|V(G)|-k\le |R_1\cap V(G)|\le |N(R_2)\cap V(G)|\le k$. Now, either $|R_2\cap V(G)|\le k$ and $W_1\cap V(G), R_2\cap V(G)$ is a winning strategy on $G$ with at most $k$ hunters, or $|R_2\cap V(G)|>k\ge |V(G)|-k$ and $N(R_2)\cap V(G),V(G)\setminus R_2$ is a winning strategy on $G$ with at most $k$ hunters. This can be easily seen by observing that there is no edge joining vertices of $V(G)\setminus N(R_2)$ with vertices of $R_2$. \\
{\bf Case 2}: By Lemma \ref{lem:sym}, if $W'_1,W'_2,W'_3$ is a hunter winning strategy, then $W'_3,W'_2,W'_1$ is winning too. Therefore, without loss of generality, we can suppose $|C_2|=|C_3|=k-1$, which means $W'_2$ and $W'_3$ have at most one vertex of $G$. Let $J$ be the graph induced by $R_1\cap V(G)=V(G)\setminus W_1'$ and let $S$ be the subgraph obtained by deleting from $J$ the isolated vertices. Observe that $V(S)=N(R_1)\cap V(J)$. Since $W'_2$ and $W'_3$ have at most one vertex of $G$,  we have $h(S,2)\le 1$, which is possible only if $S$ (and thus $J$) does not contain two disjoint edges nor a triangle. We can assume $J$ contains at least one edge (otherwise clearly $h(G,2)\le k$), thus $J$ is a star ($S$) plus possibly some isolated vertex. If $J$ contains more than one edge, let $c$ be the center of the star $S$ and consider a vertex $v\in W'_1\cap V(G)$. If $v$ has a neighbor $w$ in $J$, then $w=c$, otherwise, since $c$ is adjacent in $J$ to some $z\neq w$, we would have $v,z,c\in N(R_1)\cap V(G)$, and so two among $w,z,c$ belong to $N(R_2)\cap V(G)$, which is not possible. If $J$ has only one edge $cz$, then by the same argument we see that no vertex of $W'_1\cap V(G)$ can be adjacent to any $w\neq z,c$. Moreover if we have $v_1,v_2 \in W'_1\cap V(G)$ with $v_1$ adjacent to $c$ and $v_2$ adjacent to $z$, then $v_1,v_2,z,c\in N(R_1)\cap V(G)$, and so at least two among $v_1,v_2,z,c$ belong to $N(R_2)\cap V(G)$, which is not possible. Thus all vertices in $W'_1\cap V(G)$ that are adjacent to $J$, are adjacent to the same edge endpoint, we can assume it is $c$ (by possibly renaming the edge endpoints). In summary, regardless of how many edges $J$ has, all vertices of $W'_1\cap V(G)$ adjacent to $J$ are adjacent to the same vertex $c$. Now take $v\in W'_1\cap V(G)$ (at least one such vertex exists, otherwise trivially $h(G,2)\le k$), define $W_1=((W'_1\cap V(G))\cup \lbrace c \rbrace ) \setminus \lbrace v \rbrace $. By the previous arguments we can see that $V(G)\setminus W_1$ is an independent set, thus $W_1,W_1$ is a winning strategy with less than $k$ hunters in $G$. \\
{\bf Case 3}: In this case the total number of hunter shots on $V(G)$ is $4$. This means that the matching number of $G$ is at most $2$. To see this, observe that if there is a matching of size $3$, then at least six shots are required: two shots per disjoint edge. 
Using Proposition \ref{matchingvc}, we obtain that $h(G,2)\le 4$ and since we assumed $k \geq 4$,  $h(G,2)\le k$ holds.  \\
This completes our proof.
 \end{proof}

 Let $l,k$ be integers with $k>2$. We now describe  a polynomial reduction from the calculation of hunting number with time limit $l$ to the calculation of hunting number with time limit $l+2$. Consider a graph $T$ obtained by fully connecting a clique $K_k$ with an independent set $I$ on $k+1$ vertices. Let $G''$ be the graph obtained as the disjoint union of $G$ and $T$.
 \begin{lemma}\label{lplus2}
  Let $k,l,G,G''$ be as above. We have $h(G,l)\le k$ if and only if $h(G'',l+2)\le k$
 \end{lemma}
 \begin{proof}
 Let $W_1,...,W_l$ be a hunter winning strategy  on $G$ with at mos $k$ hunters, then $W_1,...,W_l,V(K_k),V(K_k)$ is clearly   a hunter winning strategy on $G''$, also using  at most $k$ hunters.\\
 Conversely let $W''_1,...,W''_{l+2}$ be a hunter winning strategy on $G''$ that minimizes the total number of shots.\\
 {\bf Observation ($\ast$)} If some $R_i$ contains two vertices of $K_k$ or a vertex of $K_k$ and a vertex of $I$, then $N(R_i)\supseteq V(T)$.\\
 Let $j$ be the first time step at which $W''_j$ contains vertices of $T$, by ($\ast$) and the minimality of the strategy, we have $W''_j=V(K_k)$. Moreover, $W''_{j+1}$ has at least $k-1$ vertices of $K_k$, because $N(R_j)\cap V(T)= V(K_k)$. It follows that either $W''_{j+1}=V(K_k)$ and $W''_i\cap V(T)=\emptyset$ for $i> j+1$, or $W''_{j+1}$ contains exactly $k-1$ vertices of $K_k$, namely all vertices of $K_k$ except one that we denote by $v$. Now $N(R_{j+1})\supseteq V(T)\setminus \lbrace v \rbrace$. By the fact that $I$ has at least two vertices, combined with ($\ast$) and the minimality of the strategy, we get that $W''_{j+2}$ must contain exactly $k-1$ vertices of $K_k$, namely all of them except $v$ (recall $v\notin N(R_{j+1})$). Now, $N(R_{j+2})\cap V(T)=V(K_k)$, therefore we are in the same conditions of step $j+1$ and we can apply the same arguments until we get to a step $j+d$, with $d\ge 1$ odd, when $T$ is decontaminated and $W''_i\cap V(T)=\emptyset$ for $i> j+d$.\\
 In summary the shots on $T$ are exactly at steps $j,j+1,...,j+r$ with $r\ge 1$ odd, $W_{j}''=W''_{j+r}=V(K_k)$ and between steps $j$ and $j+r$ there is an even (possibly $0$) number of steps where all vertices of $K_k$ except $v$ are shot.\\ 
 Now we define a hunter strategy on $G$ as follows: $W_i=W''_i$ for $1\le i\le j-1$, $W_j=W''_{j+2}\cap V(G)$, $W_{j+1}=(W''_{j+1}\cup W''_{j+3})\cap V(G)$, $W_i=W''_{i+2}$ for $j+2\le i \le l$. Observe that this strategy uses at most $k$ hunters since, when $r=1$, $|W''_{j+1}\cap V(G)|=0$; when $r\ge 3$, $|W''_{j+1}\cap V(G)|=1$ and  $|W''_{j+3}\cap V(G)|\le 1 < k-1$. Supposed $W_1,...,W_l$ is not winning on $G$, then there is a rabbit walk escaping it: $r_1,...,r_l$. Notice that $r_j\notin W_{j}''\cup W_{j+2}''$ because $W_{j}''=V(K_k)$ and $W_{j+2}''\cap V(G)=W_j$, moreover $r_{j+1} \notin (W''_{j+1}\cup W''_{j+3})$ by definition of $W_{j+1}$. It follows that the walk $r_1,...,r_{j-1},r_j,r_{j+1},r_j,r_{j+1},r_{j+2},r_{j+3},...,r_l$ is escaping from $W_1'',..., W_{l+2}''$ in $G''$, which is a contradiction. Therefore we can conclude $W_1,...,W_l$ is a hunter winning strategy on $G$.
 \end{proof}
 Combining the previous lemmas, we can get the main result of this section.
 \begin{theorem}
 For a given graph $G$ and positive integers $l,k$ (with $l$  polynomial in $|V(G)|$), deciding whether $h(G,l)\le k$ is NP-complete.
 \end{theorem}
 \begin{proof}
 If $l$ is even, one can get a reduction from the problem of deciding whether $h(G,2)\le k$ by $\frac{l-2}{2}$ applications of the construction in Lemma \ref{lplus2}. Similarly, if $l$ is odd, one can get a reduction from the problem of deciding whether $h(G,3)\le k$ by $\frac{l-3}{2}$ applications of the construction in Lemma \ref{lplus2}. Notice that the constructions arising from $O(l)$ applications of Lemma \ref{lplus2} still have polynomial size. Therefore, combining Propositions \ref{prop:hG2} and \ref{2to3} with Lemma \ref{lplus2} we obtain NP-hardness of deciding whether $h(G,l)\le k$.  \\
 Finally, to see that the problem belongs to NP, observe that by our assumption on $l$, one can verify in polynomial time whether a given strategy $W_1,...,W_l$ is hunter-winning. 
 \end{proof}
 
 Let us now present a simple $l$-factor approximation algorithm. As done in \cite{BENAMEUR2022}, given a graph $G=(V,E)$, we build a layered directed graph $L(G,l)$ defined as follows.   $L(G,l)$ has $l$ layers, where each layer contains a copy of all vertices of $G$, in addition to a source $s$ connected to all the vertices of the first layer, and a sink $t$ connected from all the vertices of the last layer. For a vertex $v \in V$, its copy belonging to the i-th level is denoted $v^i$. 
 For each edge $uv \in E$  and each $i \in [l-1]$, $L(G,l)$ contains the arcs  $(u^i,v^{i+1})$ and  $(v^i,u^{i+1})$. 
  More formally, the layered graph $L(G,l)$ has vertex set $\{v^i : v \in V, \ i \in [l]\}\cup\{s,t\}$  and arc set $\{(u^i,v^{i+1}), (v^i,u^{i+1}) \mbox{ for all }  uv \in E, \ i \in [l-1]\} \cup \{(s,v^1),(v^l,t) \mbox{ for all } v \in V\}$.  Observe that a walk of length $l$ in $G$ corresponds to an $s-t$ directed path in $L(G,l)$. Considering $l$ layers in $L(G,l)$ is equivalent to imposing a capture time  less than or equal to $l$.   Let $ca(G,l)$ be the minimum size of a $s-t$ vertex cut. $ca(G,l)$ can obviously be computed in polynomial time and is equal to the maximum number of internally vertex-disjoint $s-t$ paths. We show below that $ca(G,l)$ is a $\min(l,n)$-approximation of $h(G,l)$.   
 
 \begin{proposition}
 $h(G,l) \leq ca(G,l) \leq \min(l,n) \times h(G,l)$. 
 \end{proposition}
 \begin{proof}
 Since $ca(G,l)$ is the minimum size of a $s-t$ vertex-cut,  it can be seen as the minimum number of shots needed to capture the rabbit, regardless of the number of hunters. The value $ca(G,l)$ is then an upper bound of $h(G,l)$ (in the worst case, each hunter will only shoot once).   Furthermore, in an optimal winning strategy using $h(G,l)$ hunters, each hunter cannot shoot more than $l$ times, thus proving that $ca(G,l) \leq l \times h(G,l)$. Finally, $ca(G,l) \leq n$, implying that $ca(G,l) \leq n \times h(G,l)$.
 \end{proof}

\section{Characterization of $1$-hunterwin graphs with loops}

\begin{figure}[htbp]
    \centering
    \vspace{-5mm}
    \includegraphics[scale=0.3]{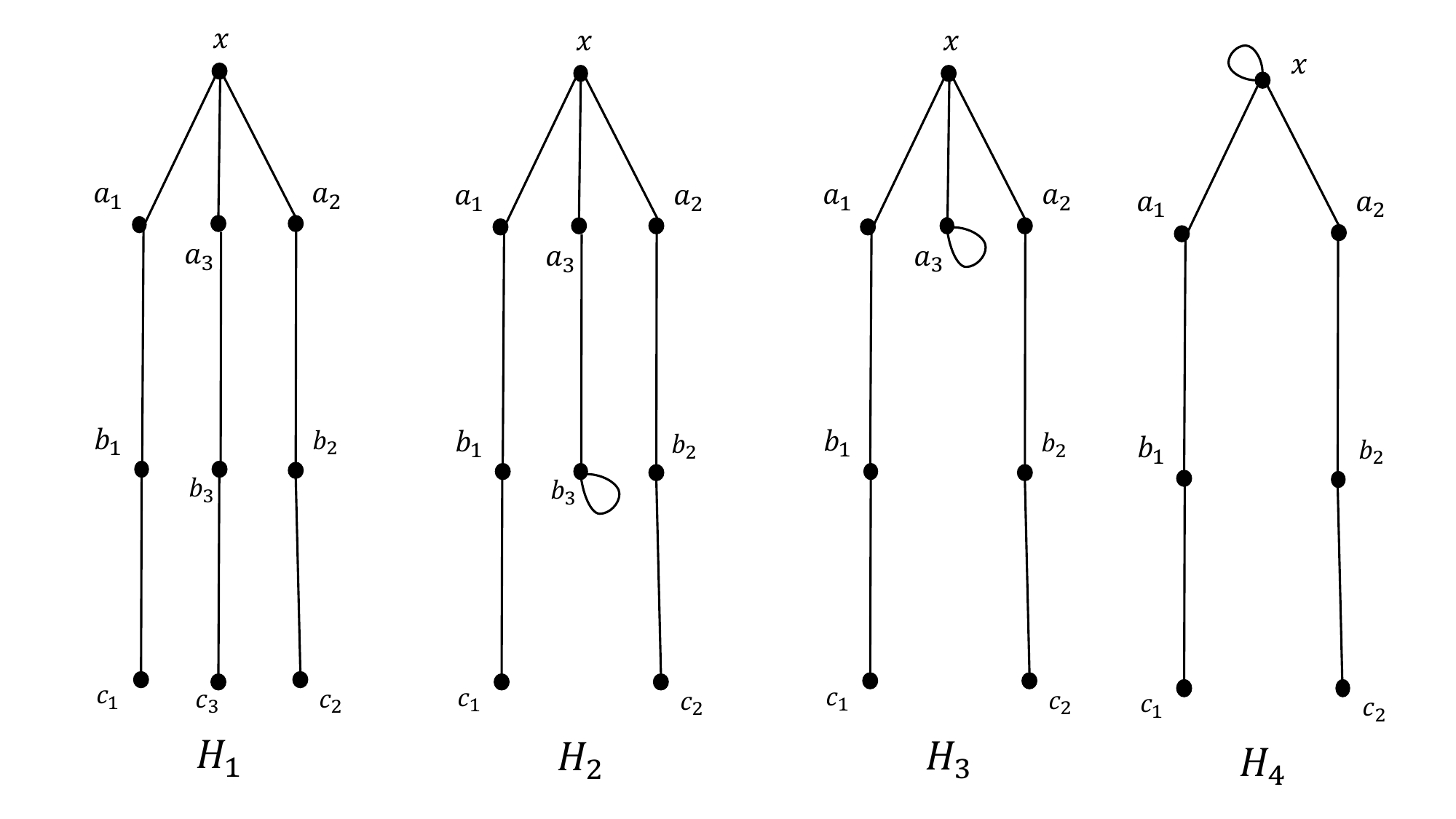}
    \caption{The $4$ forbidden graphs $H_1$, $H_2$, $H_3$ and $H_4$}
    \label{fig:forbid}
\end{figure}

We already know from \cite{benameur2024complexityresultscopsrobber}  that it is difficult to characterize directed graphs for which $h(G) = 1$. However,  a characterization is known for simple graphs \cite{journals/combinatorics/BritnellW13,HASLEGRAVE20141}.  We would like to go further by considering undirected graphs  that  might contain loops, and we aim to characterize those for which $h(G) = 1$. 
It is possible to linearly characterize such graphs using the characterization in \cite{journals/combinatorics/BritnellW13,HASLEGRAVE20141} and Lemma \ref{lem:bipartite}. Nonetheless, we provide a characterization by means of forbidden subgraphs extending the one in \cite{journals/combinatorics/BritnellW13,HASLEGRAVE20141} for loopless graphs. We will prove that  cycles, connected loops and the $4$ graphs $H_1$, $H_2$, $H_3$ and $H_4$ shown on Figure \ref{fig:forbid} are precisely the forbidden subgraphs in a 1-hunterwin graph with loops.  Notice that $H_1$ is the $3$-spider graph used in \cite{journals/combinatorics/BritnellW13,HASLEGRAVE20141} to characterize 1-hunterwin acyclic simple graphs. 

\begin{theorem}
A graph $G$ is $1$-hunterwin if and only if it does not contain cycles,  two connected loops, 
{ $H_1$,  $H_2$,  $H_3$ or $H_4$.}
\label{th:1win}
\end{theorem}
\begin{proof}
First of all let us apply the construction of Lemma \ref{lem:bipartite} and obtain the graph $B_G$. We will prove that $G$ contains a forbidden subgraph among those listed above if and only if $B_G$ contains a cycle or a $3$-spider.\\
If $G$ contains a cycle $v_1v_2...v_kv_1$, then $B_G$ contains the cycle $v_1v_2'...v_k'v_1$ if $k$ is even  
{and $v_1 v_2'...v_k v'_1 v_2 ... v'_k v_1$ }
if $k$ is odd.\\
If $G$ contains two loops on $v_1$ and $v_k$ and a path $v_1...v_k$, then $B_G$ contains the cycle $v_1v_2'...v_k'v_k...v_2v_1'v_1$ if $k$ is even and $v_1v_2'...v_kv_k'...v_2v_1'v_1$ if $k$ is odd.\\
If $G$ contains paths $c_1b_1a_1xa_2b_2c_2$ and $xa_3b_3c_3$ {(i.e., $H_1$),} then $B_G$ will contain the $3$-spider formed by $c'_1b_1a'_1xa'_2b_2c'_2$ and $xa'_3b_3c'_3$. If $G$ contains paths $c_1b_1a_1xa_2b_2c_2$ and $xa_3b_3$ with a loop on $b_3$ {(i.e., $H_2$),} then $B_G$ contains the $3$-spider formed by $c'_1b_1a'_1xa'_2b_2c'_2$ and 
{$xa'_3b_3b'_3$}. If $G$ contains paths $c_1b_1a_1xa_2b_2c_2$ and $xa_3$ with a loop on $a_3$ {(i.e., $H_3$),} then $B_G$ contains the $3$-spider formed by $c'_1b_1a'_1xa'_2b_2c'_2$ and $xa'_3a_3x'$. If $G$ contains paths $c_1b_1a_1xa_2b_2c_2$ and a loop on $x$ {(i.e., $H_4$),} then $B_G$ contains the $3$-spider formed by $c'_1b_1a'_1xa'_2b_2c'_2$ and $xx'a_1b'_2$.\\
Now suppose $B_G$ has a cycle or a $3$-spider $S$. If $S$ contains two edges $aa',bb'$, then $S$ contains a path joining at least two of the endpoints of the above edges. Such a path projects into a walk joining $a$ and $b$ in $G$, implying that $G$ contains two connected loops. We can thus suppose that $S$ contains at most one edge of the form $aa'$. If $S$ does not contain two copies $v,v'$, then its projection on $G$ is a cycle (resp. $3$-spider), since $S$ is a cycle (resp. $3$-spider). If $S$ contains two copies of the same vertex, let $v,v'$ be a pair of such copies that are closest in $S$: if $vv'$ is not an edge of $S$, then a shortest path of $S$ joining $v$ and $v'$ has length at least 3 and does not contain two copies of the same vertex, thus it projects into a cycle of $G$. We can thus assume in what follows that $S$ contains an edge $vv'$ and no other edge of this form. \\
If $S$ is a cycle, let $S'$ be the subpath of $S$ joining $v$ and $v'$ having length at least 3: if $S'$ does not contain copies of the same vertex, then it projects into a cycle of $G$; if $S'$ contains copies of the same vertex, let $w,w'$ be a pair of such copies that are closest in 
{$S'$}: there must exist a path in 
{$S'$}
joining $w$ and $w'$ of length at least 3 not containing copies of the same vertex, this path projects into a cycle of $G$.\\
If $S$ is a $3$-spider, let $P$ be the (possibly empty) path in $S$ joining {the center of the spider $x$ and the endpoint of the edge $vv'$ closest to $x$ (note that the edge $vv'$ does not belong to $P$)}. {Let $L_1,L_2$ be the two legs of the spider that do not contain $vv'$}. {If in $P\cup L_1\cup L_2$} there are no copies of the same vertex, then the projection of {$P\cup L_1\cup L_2\cup vv'$ into $G$ is} a path $c_1b_1a_1xa_2b_2c_2$ and a path $xa_3v$ with a loop on $v$ or a path $xv$ with a loop on $v$ or a loop on $x$ (when the length of $P$ is respectively $2,1$ or 0) corresponding respectively to $H_2$, $H_3$ and $H_4$. 
{Otherwise, if there are copies of the same vertex in $P\cup L_1\cup L_2$, let $w,w'$ be the closest such pair}: there must exist a path in {$P\cup L_1\cup L_2$}, of length at least $3$, joining $w$ and $w'$ not containing copies of the same vertex, this path projects into a cycle of $G$. 
    \end{proof}

\section{Concluding remarks}

We know that deciding whether one hunter can win is polynomial, but the complexity of deciding whether the hunting number is less than some given constant $k$ is unknown for $k\ge 2$, even in the case of trees.

Deciding whether $h(G)\le k$ can be seen as a special case of the integer matrix mortality problem where, given $m$ binary square matrices of size $n$, one wants to determine whether any product obtained using these matrices results in the zero matrix.  In fact, a hunter strategy can be seen as a product of matrices by building, for each subset  $W_i$, an integer matrix obtained from the adjacency matrix of $G$ by  setting to $0$ all the entries in the columns related to the vertices in $W_i$. A strategy is then a winning one if the product of these matrices is the zero matrix. \\
Notice that the integer matrix mortality problem is PSPACE-complete (see \cite{benameur2024complexityresultscopsrobber} for references). {This implies that our problem is in PSPACE and might well be PSPACE-complete, since we have no evidence that it belongs to NP.}  {Moreover a hunter winning strategy cannot be a polynomial certificate, as in \cite{althoetmar2025} the authors propose a construction where the shortest hunter winning strategy has superpolynomial length.}  \\
Observe that our characterization of graphs for which $h(G)=1$ implies that the integer matrix mortality problem is polynomial when the input matrices are obtained by setting to $0$ all the entries in a column of the same symmetric matrix (in other words they are of the form $M(I-e_ie_i^T)$ for a symmetric matrix $M$ of size $n$ and $1\le i \le n$). {When $M$ is not necessarily symmetric, the same problem becomes NP-hard \cite{benameur2024complexityresultscopsrobber}.}

Finally, characterizing other classes of graphs for which $h(G,l)$ or $h(G)$ can be computed in polynomial time is another promising research direction.

\section*{Acknowledgments}

This research benefited from the support of the FMJH Program Gaspard Monge for optimization and operations research and their interactions with data science, the IDEX-ISITE initiative CAP 20-25 (ANR-16-IDEX-0001), the International Research Center ``Innovation Transportation and Production Systems'' of the I-SITE CAP 20-25, the ANR project GRALMECO (ANR-21-CE48-0004), and ERC grant titled PARAPATH.

\section*{Bibliography} 
\renewcommand{\bibsection}{}
\bibliographystyle{abbrvnat}
\bibliography{sample} 

\newpage
\appendix



\end{document}